\pdfoutput=1
\documentclass[a4paper]{amsart}
\usepackage{
 amsmath,
 amssymb,
 fullpage,
 mathtools,
 mathpazo,
 thmtools,
 tikz-cd,
 enumerate
}
\usepackage[
 pagebackref,
 colorlinks,
 allcolors=blue
]{hyperref}

\usepackage[capitalise, noabbrev]{cleveref}

\newcommand{\fa}{\mathfrak{a}}
\newcommand{\fb}{\mathfrak{b}}

\newcommand{\stbt}[4]{\begin{smatrix} #1 & #2 \\ #3 & #4 \end{smatrix}}
\newcommand{\sph}{\mathrm{sph}}

\newlength{\bibitemsep}
\newlength{\bibparskip}
\let\oldthebibliography\thebibliography
\renewcommand\thebibliography[1]{%
 \oldthebibliography{#1}%
 \setlength{\parskip}{\bibparskip}%
 \setlength{\itemsep}{\bibitemsep}%
}

\newenvironment{smatrix}{\left(\begin{smallmatrix}}{\end{smallmatrix}\right)}
\newcommand{\dfour}[4]{\begin{smatrix} #1 \\ &#2 \\ &&#3 \\ &&&#4 \end{smatrix}}

\newcommand{\onto}{\twoheadrightarrow}
\newcommand{\into}{\hookrightarrow}

\newcommand{\ZZ}{\mathbf{Z}}
\newcommand{\QQ}{\mathbf{Q}}
\newcommand{\CC}{\mathbf{C}}
\newcommand{\Qp}{\QQ_p}
\newcommand{\Zp}{\ZZ_p}

\newcommand{\fp}{\mathfrak{p}}
\newcommand{\fz}{\mathfrak{z}}
\newcommand{\dep}{\mathrm{dep}}
\newcommand{\Si}{\mathrm{Si}}
\newcommand{\Kl}{\mathrm{Kl}}

\newcommand{\cB}{\mathcal{B}}
\newcommand{\cO}{\mathcal{O}}
\newcommand{\cE}{\mathcal{E}}
\newcommand{\cS}{\mathcal{S}}
\newcommand{\cW}{\mathcal{W}}
\newcommand{\cL}{\mathcal{L}}

\renewcommand{\le}{\leqslant}

\renewcommand{\ge}{\geqslant}

\newcommand{\tH}{\widetilde{H}}
\newcommand{\tG}{\widetilde{G}}

\DeclareMathOperator{\Hom}{Hom}
\DeclareMathOperator{\GL}{GL}
\DeclareMathOperator{\SO}{SO}
\DeclareMathOperator{\GSp}{GSp}
\DeclareMathOperator{\GSpin}{GSpin}
\DeclareMathOperator{\ch}{ch}
\DeclareMathOperator{\Iw}{Iw}

\newtheorem{definition}{Definition}[section]
\newtheorem{theorem}[definition]{Theorem}
\newtheorem{proposition}[definition]{Proposition}

\newtheorem{corollary}[theorem]{Corollary}

\theoremstyle{remark}
\declaretheorem[name=Remark,sibling=theorem,qed={\lower-0.3ex\hbox{$\diamond$}}]{remark}

\begin{document}

\title{On some zeta-integrals for unramified representations of GSp(4)}
\author{David Loeffler}
\thanks{The author gratefully acknowledges support from the following research grants:
EPSRC Standard Grant EP/S020977/1 and ERC Consolidator Grant \#101001051 ``ShimBSD''}
\address{David Loeffler, Mathematics Institute\\
 University of Warwick\\
 Coventry CV4 7AL, UK.}
 \email{d.a.loeffler@warwick.ac.uk}

\renewcommand{\urladdrname}{\itshape ORCID}
\urladdr{\href{http://orcid.org/0000-0001-9069-1877}{0000-0001-9069-1877}}
\begin{abstract}
 This article is a companion to several works of the author and others on the arithmetic of automorphic forms for $\GSp_4$, and their associated $L$-functions and Galois representations. These works require, at various points, an input from smooth representation theory over $p$-adic local fields: the computation of values of the unique trilinear form on an unramified representation of $\GSp_4 \times \GL_2 \times \GL_2$ invariant under the diagonal $\GL_2 \times \GL_2$, for various different input data. In this note, we carry out these local computations, using a general formula for Shintani functions due to Gejima.
\end{abstract}

\maketitle


\section{Introduction}

 \subsection{Setting} We shall consider the following setting:

 \begin{itemize}
  \item $F$ is a finite extension of $\Qp$, for some prime $p$.
  \item $\cO$ denotes the ring of integers of $F$, $\fp$ the maximal ideal of $\cO$, and $q$ the cardinality of $\cO/\fp$.
  \item $\varpi$ denotes an arbitrary uniformizer of $\fp$.
  \item $|\cdot|$ denotes the absolute value on $F$, normalised by $|\varpi| = \tfrac{1}{q}$.
  \item We fix a nontrivial additive character $e: F \to \CC^\times$, which is unramified (i.e.~trivial on $\cO$ but not on $\varpi^{-1} \cO$).
  \item $G$ denotes the group $\GSp_4(F) = \{ g \in \GL_4(F): {}^tg J g = \nu(g) J\}$ where $J = \begin{smatrix} &&&1\\ &&1 \\ &-1 \\-1\end{smatrix}$. Slightly abusively we also use $J$ to denote the element $\stbt{}{1}{-1}{}$ of $\GL_2(F)$.
  \item $\tH$ denotes $\GL_2(F) \times \GL_2(F)$, and $H \subset \tH$ the group $\{ (h_1, h_2) \in \tH: \det(h_1) = \det(h_2)\}$. We consider $H$ as a subgroup of $G$ via the embedding
  \[ \iota: \left(\begin{pmatrix} a & b \\ c & d\end{pmatrix}, \begin{pmatrix} a' & b'\\ c'& d'\end{pmatrix}\right)\mapsto
  \begin{smatrix} a &&& b\\ & a' & b' & \\ & c' & d' & \\ c &&& d \end{smatrix}.
  \]
  \item $\tG$ denotes the product $G \times \tH$, with $H$ embedded diagonally via its embeddings in $G$ and $\tH$.
  \item $B_G$ denotes the upper-triangular Borel subgroup of $G$, and $N_G$ its unipotent radical; and similarly for $H$, $\tG$ etc.

  \item We denote by $P_{\Si}$ and $P_{\Kl}$ the Siegel and Klingen parabolic subgroups of $G$ containing $B_G$.
  \item $K_G$ is the maximal compact subgroup $G \cap \GL_4(\cO)$, and similarly for $H$ etc.

  \item In this paper ``representation'' will mean an admissible smooth representation on a complex vector space.
 \end{itemize}

 \subsection{Aims of this paper}

  Our goal is to study the space
  \[ \Hom_{H}(\pi \otimes \sigma_1 \otimes \sigma_2, \CC), \tag{\dag}\]
  where $\pi$ is a representation of $G$, and $\sigma_i$ are representations of $\GL_2(F)$. Here $H$ acts on $\pi \otimes \sigma_1 \otimes \sigma_2$ via its embedding into $\tG$.

  We suppose that $\pi$ is irreducible and generic. The $\sigma_i$ will be either irreducible and generic, or reducible principal-series representations with 1-dimensional quotient (so they posess Whittaker models, even in the reducible cases). In particular, all three representations have well-defined central characters; and the space $(\dag)$ is trivially 0 unless $\chi_{\pi} \chi_{\sigma_1} \chi_{\sigma_2} = 1$ as a character of $F^\times$, so we shall assume this condition holds henceforth.

  In \cref{sect:multone}, we recall (and slightly reformulate) two general theorems: one due to Moeglin and Waldspurger, showing that when $\chi_{\pi}$ is a square in the group of characters of $F^\times$, the space (\dag) has dimension $\le 1$; and another due to Kato--Murase--Sugano, which shows that if $\pi$ and the $\sigma_i$ are unramified, then (\dag) is 1-dimensional and is generated by a homomorphism which is non-zero on the spherical test data.

  In the remaining sections of the paper, we focus on the unramified case, and compute the values of a suitably normalised basis vector of (\dag) on various choices of non-spherical test data, obtaining various Euler factors. These results are used in the papers \cite{LZ20,LZvista,LZ21-erl} in order to study the arithmetic of automorphic representations of the global groups $\GSp_4$, $\GSp_4 \times \GL_2$, and $\GSp_4 \times \GL_2 \times \GL_2$ over $\QQ$.

\section{A multiplicity one result}
 \label{sect:multone}

 \subsection{Multiplicity one}

  Let $\pi$, $\sigma_1$, $\sigma_2$ be representations satisfying the conditions above.

  \begin{theorem}[Moeglin--Waldspurger]
   If $\chi_{\pi} = \omega^2$ for some $\omega: F^\times \to \CC^\times$, then the space (\dag) has dimension $\le 1$.
  \end{theorem}

  \begin{proof}
   Replacing $\pi$ with $\pi \otimes \omega^{-1}$, and one of the $\sigma_i$ with $\sigma_i \otimes \omega$, we may suppose that $\chi_{\pi} = \chi_{\sigma_1} \chi_{\sigma_2} = 1$. Thus $\pi$ factors through the quotient group $\bar{G} = G / F^\times$, which is the split special orthogonal group $\SO_5(F)$; and the $H$-representation $\sigma = (\sigma_1 \otimes \sigma_2)|_{H}$ factors through $\bar{H} = H / F^\times = \SO_4(F)$, embedded in $\SO_5(F)$ as the stabiliser of an anisotropic vector.

   If $\sigma$ is irreducible, then we are done, by the main theorem of \cite{waldspurger12}. This leaves various special cases to clean up:

   \begin{itemize}
   \item If the $\sigma_i$ are irreducible, but their tensor product becomes reducible on restricting to $H$, then $\sigma$ must be the direct sum of finitely many distinct, irreducible $H$-representations, all members of the same generic $L$-packet for $H$. (These summands are parametrised by the set of quadratic characters $\eta$ of $F^\times$ such that both of the $\sigma_i$ satisfy $\sigma_i \cong \sigma_i \otimes \eta$.) However, the Gross--Prasad conjecture for special orthogonal groups (the main theorem of \cite{moeglinwaldspurger12}) shows that the sum of the multiplicities $\dim \Hom_H(\pi \otimes \tau, \CC)$, as $\tau$ varies over a generic $L$-packet for $H$, is always 0 or 1.

   \item If one or both of the $\sigma_i$ is a reducible principal series, then $\sigma$ belongs to the class of induced representations of $\SO_4(F)$ considered in \S 1.3 of \cite{moeglinwaldspurger12}, with the group $G_0'$ being trivial.\footnote{The author is very grateful to Kei Yuen Chan (pers. comm.) for pointing out that the results of \cite{moeglinwaldspurger12} still apply to certain reducible parabolic inductions.} Hence the Proposition \emph{loc.cit.}~shows that (\dag) has the same dimension as the space of Whittaker functionals on $\pi$, which is 1.\qedhere
   \end{itemize}
  \end{proof}

  \begin{remark}
   This partially confirms a conjecture from an earlier paper of the author \cite{loeffler-zeta1}, which predicts that the above ``multiplicity $\le 1$'' statement should hold for all $\pi$, $\sigma_1$, $\sigma_2$ satisfying our running conditions, without the additional hypothesis on $\chi_{\pi}$.

   If $\sigma$ is irreducible as an $H$-representation, then the hypothesis on $\chi_{\pi}$ can be removed using the results of \cite{emorytakeda21}, which are an analogue of \cite{waldspurger12} with the special orthogonal groups replaced by spin similitude groups (note that $G \cong \GSpin_5$ and $H \cong \GSpin_4$). However, we do not know if the more refined results above also carry over to the $\GSpin$ case.
  \end{remark}

 \subsection{Unramified case}

  \begin{theorem}[Kato--Murase--Sugano]\label{thm:KMS}
   Suppose $\pi, \sigma_1, \sigma_2$ are as above, and suppose additionally that all three representations are unramified. Then $\dim \Hom_H(\pi \otimes \sigma_1 \otimes \sigma_2, \CC) = 1$, and if $v_0^{\sph}, v_1^{\sph}, v_2^{\sph}$ are nonzero spherical vectors for $\pi$ and the $\sigma_i$, then any nonzero $\fz \in \Hom_H(\pi \otimes \sigma_1 \otimes \sigma_2, \CC)$ satisfies $\fz(v_0^{\sph}, v_1^{\sph}, v_2^{\sph}) \ne 0$.
  \end{theorem}

  \begin{proof}
   Let $\sigma = (\sigma_1 \boxtimes \sigma_2)|_{H}$. If $\sigma$ is irreducible, or if the $\sigma_i$ are reducible but indecomposable, then $\sigma$ is generated by its spherical vector as an $H$-representation, so the theorem follows from the main result of \cite{katomurasesugano03}; see \cite[\S 3.7]{LSZ17} for the above formulation.

   The only remaining case is when the $\sigma_i$ are irreducible, but the tensor product becomes reducible on restriction to $H$. In this case, we have $\sigma = \tau \oplus \tau'$ for two distinct irreducible $H$-representations $\tau, \tau'$, with $\tau$ spherical and $\tau'$ non-spherical. The Kato--Murase--Sugano theorem shows that $\Hom_H(\pi \otimes \tau, \CC)$ is 1-dimensional, spanned by a form which is non-vanishing on the spherical vector. By the multiplicity result for $L$-packets quoted above, $\Hom_H(\pi \otimes \tau', \CC)$ must be zero, so the result follows.
  \end{proof}

  \begin{definition}
   Let $v_0^{\sph}, v^{\sph}_1, v^{\sph}_2$ be choices of non-zero spherical vectors of $\pi$ and of the $\sigma_i$, and let $\fz \in \Hom_H(\pi \otimes \sigma_1 \otimes \sigma_2, \CC)$ be the unique homomorphism with $\fz(v_0^{\sph}, v_1^{\sph}, v_2^{\sph}) = 1$.
  \end{definition}

  The goal of the remaining sections of this paper will be to compute the values of $\fz$ on various (non-spherical) test vectors.


\section{P-stablised vectors}

 We now define a collection of vectors in $\pi$ and $\sigma_i$ at various non-spherical levels, which arise naturally in applications to $p$-adic deformation.

 \subsection{Notation: Hecke parameters}

  \begin{definition} \
   \begin{itemize}
    \item Let $(\alpha, \dots, \delta)$ be a choice of ordering of the Hecke parameters of $\pi$ (i.e.~the Satake parameters of $\pi \otimes |\cdot|^{-3/2}$), so that $\alpha \delta = \beta\gamma\ = q^3 \chi_{\pi}(\varpi)$.
    \item Let $(\fa_1, \fb_1)$ be an ordering of the Hecke parameters of $\sigma_1$, and similarly $\sigma_2$, so that $\fa_i \fb_i = q \chi_{\sigma_i}(\varpi)$.
   \end{itemize}
   (We include the factors of $q$ and $q^{3/2}$ since these give a more convenient normalisation for Hecke operators.)
  \end{definition}

  \begin{remark}
   Thus we have 8 parameters $\{\alpha, \beta, \gamma, \delta, \fa_1, \fb_1, \fa_2, \fb_2\}$; but 2 of these are redundant, since the condition $\chi_{\pi} \chi_{\sigma_1} \chi_{\sigma_2} = 1$ tells us that
   \[
    \alpha\delta \cdot \fa_1 \fb_1 \cdot \fa_2 \fb_2 = \beta\gamma\cdot \fa_1 \fb_1 \cdot \fa_2 \fb_2 = q^5.
   \]
   So an independent set of parameters is $\{\alpha, \beta, \fa_1, \fb_1, \fa_2, \fb_2\}$, corresponding to the fact that the quotient $\tG / (H \cap Z_{\tG})$ is a reductive group of rank 6.
  \end{remark}

  \begin{definition}
   Following \cite{gejima18}, we define the Euler factors
   \begin{gather*}
    \Delta_0 = \left(1 - \tfrac{\beta}{\alpha}\middle)
    \middle(1 - \tfrac{\gamma}{\alpha}\middle)
    \middle(1 - \tfrac{\gamma}{\beta}\middle)\middle(1 - \tfrac{\delta}{\alpha}\right),
    \\
    \Delta_i = \left(1 - \tfrac{\fb_i}{\fa_i}\right)\quad (i=1,2),
    \\
    \cE =
    \left.
    \middle( 1 - \tfrac{q^2}{\alpha \fa_1 \fa_2}\middle)
    \middle( 1 - \tfrac{q^2}{\alpha \fb_1 \fa_2}\middle)
    \middle( 1 - \tfrac{q^2}{\alpha \fa_1 \fb_2}\middle)
    \middle( 1 - \tfrac{q^2}{\alpha \fb_1 \fb_2}\middle)
    \middle( 1 - \tfrac{q^2}{\beta \fa_1 \fa_2}\middle)
    \middle( 1 - \tfrac{q^2}{\beta \fa_1 \fb_2}\middle)
    \middle( 1 - \tfrac{q^2}{\beta \fb_1 \fa_2}\middle)
    \middle( 1 - \tfrac{q^2}{\gamma \fa_1 \fa_2}\middle)
    \right..
   \end{gather*}
  \end{definition}

  (These correspond to the elements denoted $\mathbf{d},\mathbf{d}'_1$, $\mathbf{d}'_2$ and $\mathbf{b}$ in \cite{gejima18}, but we have adopted a slightly different notation to avoid clashes with our Hecke parameters $\fa_i, \fb_i$.)

  \begin{definition}
   For $m, n \in \ZZ_{\ge 0}$, respectively $m \ge 0$, we define
   \[ s_{m, n} = \dfour{1}{\varpi^n}{\varpi^{m+n}}{\varpi^{m + 2n}} \in G, \qquad s_{m} = \stbt{1}{}{}{\varpi^{m}} \in \GL_2(F),\]
   and
   \[ t_{m, n} = \dfour{\varpi^{m + 2n}}{\varpi^{m+n}}{\varpi^n}{1} \in G, \qquad t_{m} = \stbt{\varpi^{m}}{}{}{1} \in \GL_2(F).\]
  \end{definition}
%
%
%

 \subsection{Whittaker models}

  Our assumptions imply that $\pi$ and the $\sigma_i$ have Whittaker models $\cW(\pi)$, $\cW(\sigma_i)$ (even in the reducible cases). Here, as in \cite{LPSZ1} and \cite{loeffler-zeta1}, we take Whittaker models for $\GSp_4$ and $\GL_2$ with respect to the characters
  \[
   \begin{smatrix} 1 & x & \star & \star\\ & 1 & y & \star \\ &&1 & -x \\ &&&1 \end{smatrix}\to e(x + y)
   \qquad\text{and}\qquad
   \stbt{1}{x}{}{1} \mapsto e(-x)
  \]
  respectively, which leads to simpler formulae for zeta-integrals. We can, and do, assume that the isomorphism between $\pi$ and its Whittaker model is fixed in such a way that $v_0^{\sph}$ maps to the unique spherical Whittaker function $W^{\sph}_0 \in \cW(\pi)$ satisfying $W_0^{\sph}(1) = 1$, and similarly for $v_1^{\sph}, v_2^{\sph}$.

 \subsection{Iwahori-level eigenvectors for \texorpdfstring{$\GL_2$}{GL(2)}}

  For $i = 1, 2$, let us define
  \[ W_{\fa_i} = \left( 1 -  \tfrac{1}{\fa_i}s_1\right)W_i^{\sph}  \in \cW(\sigma_i)^{\Iw(\fp)}, \]
  where $\Iw(\fp^n) = \{\stbt\star\star0\star \bmod \fp^n\}$ is the depth $n$ Iwahori subgroup of $\GL_2(\cO)$. The vector $v_{\fa_i}$ is an eigenvector for $U = [\Iw(\fp^n) t_1 \Iw(\fp^n)]$ with eigenvalue $\fa_i$, for any $n \ge 1$, and normalised so that $W_{\fa_i}(1) = 1$.

  We shall also need to consider eigenvectors for the dual Hecke operator $U' = [\Iw(\fp^n) s_1  \Iw(\fp^n)]$. For $n \ge 1$ we set
  \[ W'_{\fa_i}[n] = \left(\tfrac{q}{\fa_i}\right)^{n} s_n \stbt{}{-1}{1}{} W_{\fa_i} \in \cW(\sigma_i)^{\Iw(\fp^n)}.\]
  We write $W'_{\fa_i}$ for $W'_{\fa_i}[1]$. The vector $W'_{\fa_i}[n]$ is a basis of the one-dimensional $\fa_i$-eigenspace for $U'$ on $\cW(\sigma_i)^{\Iw(\fp^n)}$, for any $n \ge 1$. Moreover, the vectors $W'_{\fa_i}[n+1]$ and $W'_{\fa_i}[n]$ are compatible under the ``normalised trace'' maps $\cW(\sigma_i)^{\Iw(\fp^{n+1})} \to \cW(\sigma_i)^{\Iw(\fp^n)}$, for $n \ge 1$, given by $x \mapsto \frac{1}{q} \sum_{k \in \Iw(\fp^n)/\Iw(\fp^{n+1})} kx$.

  \begin{remark}
   Note that $W'_{\fa_i}$ is not normalised to 1 in the Whittaker model (in fact we have $W'_{\fa_i}(1) = -\frac{\fb_i}{\fa_i}$, but this does not seem to be particularly important for us).
  \end{remark}

  \begin{proposition}
   We have $W_{\fa_i} = \tfrac{1}{\fa_i}(U - \fb_i) W^{\sph}_i$ and similarly $W'_{\fa_i} = \frac{1}{\fa_i}(U' - \fb_i) W_i^{\sph}$. Consequently, if $\fa_i \ne \fb_i$ we have
   \[ W^{\sph}_i = \frac{W_{\fa_i}}{\left(1 - \tfrac{\fb_i}{\fa_i}\right)} + \frac{W_{\fb_i}}{\left(1 - \tfrac{\fa_i}{\fb_i}\right)} = \frac{W'_{\fa_i}}{\left(1 - \tfrac{\fb_i}{\fa_i}\right)} + \frac{W'_{\fb_i}}{\left(1 - \tfrac{\fa_i}{\fb_i}\right)}.\]
  \end{proposition}

 \begin{remark}[Reducible cases]
  If we suppose $\fb_i = q\fa_i$, then $\sigma_i$ has a codimension-one subrepresentation $\sigma_i^\circ$ which is an unramified twist of Steinberg, and $W_{\fa_i}$ is the normalised Whittaker new-vector of $\cW(\sigma_i^\circ)$. In this case we have $W_{\fa_i}'[1] = -q W_{\fa_i}$.
 \end{remark}

 \subsection{Iwahori eigenvectors for \texorpdfstring{$G$}{G}}

  Let $\Iw_G(\fp^m)$ denote the depth $m$ Iwahori subgroup of $K_G$ (consisting of elements whose reductions mod $\fp^m$ are upper-triangular). For any $m \ge 1$, on the $\Iw_G(\fp^m)$-invariants of $\pi$ we have the operators
  \begin{align*}
   U_1 &= \left[ \Iw_G(\fp^m) t_{1, 0} \Iw_G(\fp^m)\right] &
   U'_1 &= \left[ \Iw_G(\fp^m) s_{1, 0} \Iw_G(\fp^m)\right] \\
   U_2 &= \left[ \Iw_G(\fp^m) t_{0, 1}\Iw_G(\fp^m)\right]&
   U_2' &= \left[ \Iw_G(\fp^m) s_{0, 1}\Iw_G(\fp^m)\right].
  \end{align*}
  For $m \ge 1$, the maps $U_1$ and $U_2$ commute with the natural inclusion $\pi^{\Iw_G(\fp^m)} \into \pi^{\Iw_G(\fp^{m+1})}$, and $U_1', U_2'$ commute with the normalised trace $\pi^{\Iw_G(\fp^{m+1})} \onto \pi^{\Iw_G(\fp^{m})}$. At level $\Iw(\fp)$, the eigenvalues of $U_1$ or $U_1'$ are $\{ \alpha, \beta, \gamma, \delta\}$, and those of $U_2$ or $U_2'$ are $\{ \tfrac{\alpha\beta}{q},  \tfrac{\alpha \gamma}{q}, \tfrac{\beta\delta}{q}, \tfrac{\gamma\delta}{q}\}$.

  \begin{definition}
   We let $W^{\Iw}_{\alpha\beta}$ denote the vector
   \[
    \left(1 - \frac{\alpha\gamma}{q U_2}\right) \left(1 - \frac{\beta}{U_1}\right)\left(1 - \frac{\gamma}{U_1}\right)\left(1 - \frac{\delta}{U_1}\right) W_0^{\sph} \in \cW(\pi)^{\Iw_G(\fp)},
   \]
   which is a basis of the 1-dimensional subspace of $\cW(\pi)^{\Iw_G(\fp)}$ on which $U_1 = \alpha$ and $U_2 = \tfrac{\alpha\beta}{q}$.
  \end{definition}

  As in the $\GL_2$ case, our normalisations are chosen so that $W_{\alpha\beta}^{\Iw}(1) = 1$. (This is far from obvious \emph{a priori}, but can be checked explicitly from the Casselman--Shalika formula for the spherical Whittaker function.)

  \begin{remark}
   If the Hecke parameters of $\pi$ are pairwise distinct, then $W^{\Iw}_{\alpha\beta}$ and its analogues for the other seven orderings of the parameters are a basis of the 8-dimensional space $\pi^{\Iw_G(\fp)}$, and we can recover the spherical vector by the formula
   \[
    W^{\sph}_0 = \sum_{w_0} \left(\tfrac{1}{\Delta_0}\cdot W^{\Iw}_{\alpha \beta}\right)^{w_0},\qquad \Delta_0 = \left(1 - \tfrac{\beta}{\alpha}\middle)\middle(1 - \tfrac{\gamma}{\alpha}\middle)\middle(1 - \tfrac{\gamma}{\beta}\middle)\middle(1 - \tfrac{\delta}{\alpha}\right).
   \]
   Here $w_0$ varies over the Weyl group, acting via formal permutations of the parameters (this is a notational device, rather than an action on $\pi$ itself).
  \end{remark}

 \subsection{Parahoric eigenvectors}

  We also consider analogues for the Siegel and Klingen parahoric subgroups. We use the same symbol $U_1$ for the double coset of $t_{1, 0}$ at Siegel parahoric level $\Si(\fp^m)$ for any $m \ge 1$; since this operator is compatible under pullback with its namesake at Iwahori level, this should cause no confusion. Likewise, we denote the double coset of $t_{0, 1}$ at Klingen parahoric level still by $U_2$. We then define
  \begin{align*}
   W^{\Si}_{\alpha} &\coloneqq \left(1 - \frac{\beta}{U_1}\right)\left(1 - \frac{\gamma}{U_1}\right)\left(1 - \frac{\delta}{U_1}\right) W_0^{\sph}\quad \in \cW(\pi)^{\Si(\fp)}
   \\ &= \frac{1}{\left(1 - \frac{\gamma}{\beta}\right)} W^{\Iw}_{\alpha\beta} + \frac{1}{\left(1 - \frac{\beta}{\gamma}\right)}W^{\Iw}_{\alpha \gamma},
  \end{align*}

  and similarly
  \begin{align*}
   W^{\Kl}_{\alpha\beta}
   &\coloneqq \tfrac{1}{\left(1 + \tfrac{\gamma}{\alpha}\right)}\left(1 - \tfrac{\alpha\gamma}{q U_2}\right)\left(1 - \tfrac{\beta\delta}{q U_2}\right)\left(1 - \tfrac{\gamma\delta}{q U_2}\right) W_0^{\sph}\quad\in \cW(\pi)^{\Kl(\fp)}\\
   &= \frac{1}{\left(1 - \frac{\beta}{\alpha}\right)} W^{\Iw}_{\alpha\beta} + \frac{1}{\left(1 - \frac{\alpha}{\beta}\right)} W^{\Iw}_{\beta\alpha}.
  \end{align*}
  As before, the Whittaker functions $W^{\Kl}_{\alpha\beta}$ and $W^{\Si}_{\alpha}$ are normalised to be 1 at the identity, although this is not entirely obvious from the definitions.

  \begin{remark}
   The formula we have given as the definition of $W_{\alpha\beta}^{\Kl}$ does not make sense if $\gamma = -\alpha$; so we shall assume whenever the Klingen eigenvectors are considered that we are not in this case.
  \end{remark}

 \subsection{Dual parahoric eigenvectors}

  We shall also need the ``dual'' Klingen vector defined for $n \ge 1$ by
  \[ W^{\prime, \Kl }_{\alpha\beta}[n] = \left(\tfrac{q^5}{\alpha\beta}\right)^n \operatorname{pr}_{\Kl(\fp^n)} \left(s_{0, n} J \cdot W^{\Kl}_{\alpha\beta}\right) \in \cW(\pi)^{\Kl(\fp^n)},
  \]
  where $\operatorname{pr}_{\Kl(\fp^n)}$ denotes the normalised trace map to Klingen level. These vectors are compatible under the normalised trace maps, and they are all $U_{2}'$-eigenvectors of eigenvalue $\tfrac{\alpha\beta}{q}$. As usual we write $W^{\prime, \Kl }_{\alpha\beta} = W^{\prime, \Kl }_{\alpha\beta}[1]$. One can define similarly dual Iwahori eigenvectors
  \[ W^{\prime, \Iw}_{\alpha\beta}[n] = \left(\tfrac{q^8}{\alpha^2\beta}\right)^n \operatorname{pr}_{\Iw(\fp^n)}\left(s_{n, n} J \cdot W^{\Iw}_{\alpha\beta}\right)\, \mathrm{d}\gamma,
  \]
  which are $U_1' = \alpha$ and $U_2' = \tfrac{\alpha\beta}{q}$ eigenvectors. (There are also Siegel versions, but we shall not need these.) In each case, these can be characterised by the fact that their normalised traces down to the opposite parahoric levels $\overline{\Iw}(\varpi)$ etc are the images of $W_{\alpha \beta}^{\Iw}$ under $J$.

  \begin{proposition}\label{prop:klingentrace}
   We have
   \[
    \sum_{k \in K_G / \Kl(\fp)} k \cdot W^{\prime, \Kl}_{\alpha\beta} = \sum_{k \in K_G / \Kl(\fp)} k \cdot W^{\Kl}_{\alpha\beta} = q^3 \left.
    \middle(1 - \tfrac{\gamma}{q\beta}\middle)
    \middle(1 - \tfrac{\delta}{q\alpha}\middle)
    \middle(1 - \tfrac{\delta}{q\beta}\middle)
    \right. W^{\sph}_0.
   \]
  \end{proposition}

  \begin{proof}
   It suffices to prove this formula for $\gamma \ne -\alpha$ (since the remaining cases follow by analytic continuation). We first establish the formulae
   \[ W^{\Kl}_{\alpha\beta} = h_{\alpha\beta}(U_2) \cdot W^{\sph}, \qquad W^{\prime, \Kl}_{\alpha\beta}[1] = h_{\alpha\beta}(U_2')\cdot W^{\sph}\]
   where $h_{\alpha\beta}$ is the polynomial $\tfrac{1}{(1 + \tfrac{\gamma}{\alpha})}\left(1 - \tfrac{\alpha\gamma}{q X}\right)\left(1 - \tfrac{\beta\delta}{q X}\right)\left(1 - \tfrac{\gamma\delta}{q X}\right)$. The first of these is an explicit computation using the Casselman--Shalika formula; the second follows by a careful analysis of the endomorphism of $\cW(\pi)^{\Kl(\fp)}$ used to define $W^{\prime, \Kl}_{\alpha\beta}[1]$.

   By a rather lengthy but elementary computation, we can write down the matrices of the Hecke operators $U_2$ and $U_2'$ explicitly in the basis of $\pi^{\Kl(\fp)}$ given by the four double cosets $P_{\Kl} \backslash G / \Kl(\fp)$. We can therefore write $W^{\prime, \Kl}_{\alpha\beta}[1]$ and $W^{\Kl}_{\alpha\beta}$ in terms of this basis, and hence compute their projections to spherical level.
  \end{proof}

  \begin{remark}
   This computation appears in \cite{genestiertilouine05}, but their formula has $\left(1 - \frac{\gamma}{\beta}\right)$ etc in place of $\left(1 - \frac{\gamma}{q\beta}\right)$. We believe the statement given above to be the correct one.
  \end{remark}

\section{Iwahori-level Shintani functions}

 \subsection{Open orbits}

  \begin{proposition}
   Let
   \( \eta = \begin{smatrix}
        1 & 1 & 1 &   \\
          & 1 &   & \phantom{-}1 \\
          &   & 1 & -1\\
          &   &   & \phantom{-}1
       \end{smatrix} \in N_G \cap K_G.
   \)
   Then:
   \begin{enumerate}[(a)]
    \item The product $B_H \eta \overline{B}_G$ is an open subscheme of $G$ over $\cO$.
    \item $(\eta J, 1, 1) \in K_{\tG}$ represents the unique open $H$-orbit on the flag variety $\tG / B_{\tG}$.
   \end{enumerate}
  \end{proposition}

  \begin{proof}
   Part (a) is easily verified; we know that the big cell $B_G \overline{B}_G$ is open in $G$, so it suffices to show that the image of $B_H \eta$ in $B_G \overline{B}_G / \overline{B}_G \cong N_G$ is open. This is easily checked explicitly (the image is precisely the elements of $N_G$ whose top row has the form $(1, u,v, \star)$ with $u, v \ne 0$). This gives (a), and (b) follows easily from this.
  \end{proof}

  \begin{corollary}
   Choose any vectors $W_0 \in \cW(\pi)^{\Iw_G(\fp)}$ and similarly $W_i \in \cW(\sigma_i)^{\Iw_{\GL_2}(\fp)}$. Let $u = (u_0, u_1, u_2) \in \tG(\cO)$, and $m, n, x_1, x_2 \in \ZZ_{\ge 0}$. Then the quantity
   \[ \fz\left(u_0 s_{m, n} W_0, u_1 s_{x_1} W_1, u_2 s_{x_2} W_2\right) \]
   depends only on the class of $u$ in $H \backslash \tG / (B_{\tG} \cap K_{\tG})$. If $u$ lies in the $\cO$-points of the open $H$-orbit on $\tG/ B_{\tG}$, then we have
   \[ \fz\left(u_0 s_{m, n} W_0, u_1 s_{x_1} W_1, u_2 s_{x_2} W_2\right) = q^{-(x_1 + x_2 + 3m + 4n)} \fz\Big(u_0 (U'_1)^m (U'_2)^n W_0, u_1 (U')^{x_1}  W_1, u_2 (U')^{x_2} W_2\Big). \]
  \end{corollary}

  \begin{proof}
   The first statement is elementary: if $s = \left(s_{m, n}, s_{x_1} , s_{x_2} \right) \in \tG$, then $(W_0, W_1, W_2)$ is invariant under the action of the group
   \[ s \Iw_{\tG}(\fp) s^{-1} \supseteq s B_{\tG}(\cO) s^{-1} \supseteq B_{\tG}(\cO).\]

   For the second statement, it suffices to note that the coset representatives appearing in the Hecke operator are elements of $\tG(\cO)$ which are congruent to 1 modulo $\fp$, and hence translation by these elements must preserve the open orbit $H(\cO) u B_{\tG}(\cO) \subseteq \tG(\cO)$. So $\fz\Big(u_0 (U'_1)^m (U'_2)^n W_0, u_1 (U')^{x_1}  W_1, u_2 (U')^{x_2} W_2\Big)$ is a sum of $q^{x_1 + x_2 + 3m + 4n}$ terms all of which are equal to $\fz\left(u_0 s_{m, n} W_0, u_1 s_{x_1} W_1, u_2 s_{x_2} W_2\right)$.
  \end{proof}

  In particular, for any $m, n, x_1, x_2 \ge 0$, we have
  \begin{equation}
   \label{shintani}
   \fz\Big(\eta J s_{m, n} W_{\alpha\beta}^{\prime, \Iw}, s_{x_1} W'_{\fa_1}, s_{x_2} W'_{\fa_2}\Big) = \frac{\alpha^m (\alpha\beta)^n \fa_1^{x_1} \fa_2^{x_2}}{q^{(3m + 5n + x_1 + x_2)}}\cdot  \fz\Big(\eta J  W_{\alpha\beta}^{\prime,\Iw}, W'_{\fa_1}, W'_{\fa_2}\Big).
  \end{equation}

 \subsection{Gejima's formula}

  \begin{theorem}[Gejima]\label{gejima}
   For any $m, n, x_1, x_2 \in \ZZ_{\ge 0}$ with $x_1 - x_2$ even, we have
   \[
    \fz\left(\eta J s_{m, n} \cdot W_0^{\sph},
    s_{x_1}  W_1^{\sph}, s_{x_2} W_2^{\sph}\right) =
   \frac{q^4}{(q^2 - 1)^2}\cdot \sum_{(w_0, w_1, w_2)}\left(\frac{\alpha^m (\alpha\beta)^n \fa_1^{x_1} \fa_2^{x_2}}{q^{(3m + 5n + x_1 + x_2)}}\cdot \frac{\cE}{\Delta_0\Delta_1 \Delta_2}\right)^{(w_0, w_1, w_2)},
   \]
   where $w_i$ are elements of the Weyl groups (acting by formally permuting the Hecke parameters). In particular, the rational function on the right (for fixed $m, n, x_1, x_2$) is a polynomial in $\{\alpha, \beta, \fa_1, \fa_2, \fb_1, \fb_2\}$ and their inverses.
  \end{theorem}

  \begin{proof}
   This is the main theorem of \cite{gejima18}, in slightly modified notation. The assumption that $x_1 - x_2$ be even implies that the class of $(\eta J s_{m, n}, s_{x_1}, s_{x_2})$ in $H \backslash \tG / (K_{\tG} Z_{\tG})$ has a representative in
   \[ H \backslash (G \times H) / K_{(G \times H)} Z_{(G \times H)} \cong K_H \backslash G / K_G Z_G, \]
   and our elements $(\eta J s_{m, n}, s_{x_1}, s_{x_2})$ for $m, n, x_1, x_2 \ge 0$ map to Gejima's coset representatives $\{ t(\lambda') \eta t(\lambda) : \lambda \in \Lambda^+, \lambda' \in \Lambda_0^{++}\}$ modulo the action of $Z_G$ (after correcting for a minor typographical error in \emph{op.cit.}, and the fact that Gejima uses a slightly different matrix model of $\GSp_4$).
  \end{proof}

  \begin{proposition}
  \label{prop:pstabShintani}
   \begin{enumerate}
   \item For any $m, n, x_1, x_2 \ge 0$, we have
      \[
       \fz\left(\eta J s_{m, n} \cdot W_{\alpha\beta}^{\prime, \Iw},
       s_{x_1}  W'_{\fa_1}, s_{x_2} W'_{\fa_2}\right) =
       \frac{q^4}{(q^2 - 1)^2}\cdot \frac{\alpha^m (\alpha\beta)^n \fa_1^{x_1} \fa_2^{x_2}}{q^{(3m + 5n + x_1 + x_2)}}\cdot \cE.
      \]
     \item The formula of \cref{gejima} holds for all $m, n, x_1, x_2 \ge 0$ (whatever the parity of $x_1 - x_2$).
   \end{enumerate}
  \end{proposition}

  \begin{proof}
   Let us first assume that the Hecke parameters of $\pi$ and the $\sigma_i$ are distinct, and that the ratios $\fa_1/\fb_1$ and $\fa_2/\fb_2$ are not both $-1$. Then the 32 functions on the set $\{(m, n, x_1, x_2)\in \ZZ_{\ge 0}^4: x_1 = x_2 \bmod 2\}$ defined by $(m, n, x_1, x_2) \mapsto \frac{\alpha^m (\alpha\beta)^n \fa_1^{x_1} \fa_2^{x_2}}{q^{(3m + 5n + x_1 + x_2)}}$, for each of the possible orderings of the Hecke parameters, are linearly independent. Moreover, the spherical vectors and Iwahori vectors are related by
   \[ W_0^{\sph} \otimes W_1^{\sph} \otimes W_2^{\sph} = \sum_{(w_0, w_1, w_2)} \left(\frac{W_{\alpha\beta}^{\prime, \Iw} \otimes W'_{\fa_1} \otimes  W'_{\fa_2}}{\Delta_0 \Delta_1 \Delta_2}\right)^{(w_0, w_1, w_2)}\]
   where the sum is over elements of the Weyl group (acting via formal permutations of the parameters). So Gejima's formula implies that we must have
   \[ \fz\left(\eta J W_{\alpha\beta}^{\prime, \Iw},
    W'_{\fa_1}, W'_{\fa_2}\right) = \frac{q^4}{(q^2 - 1)^2}\cdot \cE,
   \]
   and part (1) of the theorem follows. We can now recover part (2) of the theorem by summing over the Weyl-group orbit again, since we know that \cref{shintani} holds without the parity assumption.

   In the exceptional cases where the ratios of the Hecke parameters degenerate, we can use an analytic continuation argument: for fixed values of $(m, n, x_1, x_2)$, both sides of the claimed formulae are polynomials in $(\alpha, \beta, \fa_1, \fb_1, \fa_2, \fb_2)$ and their inverses, and we have seen that the two sides agree at a Zariski-dense set, so the result follows for all values of the parameters.
  \end{proof}

  \begin{remark}
   It seems rather likely that the elements $(\eta J s_{m, n}, s_{x_1}, s_{x_2})$ for $m, n, x_1, x_2 \ge 0$ are a set of coset representatives for $H \backslash \tG / K_{\tG} Z_{\tG}$, generalising Theorem 3.2.1 of \emph{op.cit.}. However, we shall not prove this here, since it is not needed for our purposes.

   The proof of the above corollary is actually a little redundant, since the argument of \cite{gejima18} in fact establishes a formula for a ``$p$-stabilised'' Shintani function as an intermediate step. However, since the necessary translations between conventions are somewhat awkward, and the case $x_1 - x_2$ odd is not covered in \emph{op.cit.}, we have found it easier to deduce our result formally from the statement of Gejima's result.
  \end{remark}

  \begin{corollary}
   \label{cor:pstabvariants}
   All of the following quantities are equal to $\frac{q^4}{(q^2 - 1)^2}\left(\tfrac{\alpha}{q^3}\right)^{m} \left(\tfrac{\alpha\beta}{q^5}\right)^{n} \left(\tfrac{\fa_1}{q}\right)^{x_1}\left(\tfrac{\fa_2}{q}\right)^{x_2}\cE$:

   \begin{enumerate}[(i)]
    \item $\fz\left( J \eta J s_{m,n} W_{\alpha\beta}^{\prime, \Iw}, t_{x_1} W_{\fa_1}, t_{x_2} W_{\fa_2}\right)$, for all $m, n \ge 0$ and $x_1, x_2 \ge 1$.

    \item $\fz\left(\eta t_{m,n} W_{\alpha\beta}^{\Iw}, s_{x_1} W'_{\fa_1}, s_{x_2}W'_{\fa_2}\right)$ for $m, n \ge 1$ and $x_1, x_2 \ge 0$.

    \item $\fz\left(J \eta t_{m,n} W_{\alpha\beta}^{\Iw}, t_{x_1} W_{\fa_1}, t_{x_2}W_{\fa_2}\right)$ for $m, n, x_1, x_2 \ge 1$.

    \item $\fz\left(\eta t_{m,n} W_{\alpha\beta}^{\Iw}, s_{x_1} W'_{\fa_1}, J t_{x_2}W_{\fa_2}\right)$ for $m, n, x_2 \ge 1$ and $x_1 \ge 0$.
   \end{enumerate}
  \end{corollary}

  \begin{proof}
   This follows from the previous formula by substituting formulae for the ``primed'' eigenvectors in terms of the ``non-primed'' ones.
  \end{proof}

 \subsection{Special cases}

  We now record the cases of the above formula which are relevant in Euler-system theory.\medskip

  \subsubsection{Iwahori level} Firstly, if we take $(m, n, x_1, x_2) = (1, 1, 1, 1)$ in \cref{cor:pstabvariants}(iii), we obtain the formula
  \begin{equation}
   \label{eq:iwahori1}
   \fz\left(J \eta t_{1, 1} \cdot W^{\Iw}_{\alpha\beta},
   t_1  W_{\fa_1}, t_1 W_{\fa_2}\right) =
   \frac{\alpha^2 \beta \fa_1 \fa_2}{q^6(q^2 - 1)^2} \cE.
  \end{equation}
  This is the Euler factor which arises when applying the machinery of \cite{loeffler-spherical} to construct families of $p$-adic cohomology classes on Shimura varieties for $\tG$, interpolating algebraic cycles in region $(e)$ of \cite[Figure 1]{LZvista}.

  Similarly, using part (ii) of the corollary and the definition of $W'_{\fa_1}[x_i]$ we find that
  \begin{equation}
   \label{eq:iwahori2}
   \fz\left(\eta t_{1, 1} \cdot W^{\Iw}_{\alpha\beta}, W'_{\fa_1}[x_1], W'_{\fa_2}[x_2]\right)
   = \frac{\alpha^2 \beta}{q^4(q^2 - 1)^2} \cE,
  \end{equation}
  for any $x_1, x_2 \ge 1$. This Euler factor arises naturally when considering variation of Euler systems for $\GSp_4$ in four-variable families; see forthcoming work of Rockwood. (See also \cite[Remark 17.3.10]{LZ20b-regulator} for a situation where such a formula \emph{should} have been used, but wasn't!).

  Finally, the computation we need for $\GSp_4 \times \GL_2$ is the following:
  \begin{equation}
   \label{eq:iwahori3}
   \fz\left(\eta t_{1, 1} \cdot W^{\Iw}_{\alpha\beta}, W'_{\fa_1}[x_1], J t_1 W_{\fa_2}\right)
   = \frac{\alpha^2 \beta \fa_2}{q^5(q^2 - 1)^2} \cE.
  \end{equation}

  \subsubsection{Siegel level} We also consider an analogous problem with Siegel-parahoric eigenvectors in place of Iwahori eigenvectors. Consider the quantity
  \[
     \fz\left(\eta t_{1, 0} \cdot W^{\Si}_{\alpha}, W'_{\fa_1}, W'_{\fa_2}\right).
  \]
  Although $W^{\Si}_{\alpha}$ is a linear combination of $W^{\Iw}_{\alpha \beta}$ and $W^{\Iw}_{\alpha\gamma}$, we cannot immediately reduce to the situation above, since part (ii) of the corollary above is not valid for $n = 0$. However, we can get around this problem by repeating the argument of \cref{prop:pstabShintani} with the Siegel-parahoric eigenvectors in place of their Iwahori cousins, and assuming $n = 0$ throughout. The conclusion is that for $m \ge 1$ we have
  \[
   \fz\left(\eta t_{m, 0} \cdot W^{\Si}_{\alpha}, s_{x_1} W'_{\fa_1}, s_{x_2} W'_{\fa_2}\right) =
   \frac{q^4}{(q^2 - 1)^2}\cdot \frac{\alpha^m \fa_1^{x_1} \fa_2^{x_2}}{q^{(3m + x_1 + x_2)}} \cdot \left(\frac{\cE}{(1 - \frac{\gamma}{\beta})} + \dots\right)
  \]
  where $\dots$ indicates the same term with $\beta$ and $\gamma$ interchanged. After a little computation this simplifies to
  \begin{multline}
   \label{eq:siegel}
   \fz\left(\eta t_{1, 0} \cdot W^{\Si}_{\alpha}, W'_{\fa_1}, W'_{\fa_2}\right) = \\ \frac{\alpha}{(q-1)(q+1)^2} \left.
      \middle( 1 - \tfrac{q^2}{\alpha \fa_1 \fa_2}\middle)
      \middle( 1 - \tfrac{q^2}{\alpha \fb_1 \fa_2}\middle)
      \middle( 1 - \tfrac{q^2}{\alpha \fa_1 \fb_2}\middle)
      \middle( 1 - \tfrac{q^2}{\alpha \fb_1 \fb_2}\middle)
      \middle( 1 - \tfrac{q^2}{\beta \fa_1 \fa_2}\middle)
      \middle( 1 - \tfrac{q^2}{\gamma \fa_1 \fa_2}\middle)
      \right..
  \end{multline}

%
%

 \subsection{The \texorpdfstring{$\GSp_4$}{GSp(4)} tame norm relation}
  \label{sect:gsp4tame}
  Using the above methods we can rather quickly revisit the computations of \cite[\S 3]{LSZ17}, whose goal was to compute the quantity
  \[ \fz\left( (1 - t_{1, 0}^{-1} \eta t_{1, 0}) W_0^{\sph}, W'_{\fa_1}[n], W'_{\fa_2}[n]\right), \]
  which is independent of $n$ if $n \gg 0$ (in fact $n = 2$ suffices). Since $t_1 W'_{\fa_i}[2] = \fb_i \cdot W'_{\fa_i}$, the result is
  \[ \sum_{w_0} \left[\frac{\cE}{\Delta_0}\cdot \left( 1 - \tfrac{\alpha \fb_1\fb_2}{q^3}\right)\right]^{w_0} =
  \frac{q^3}{(q-1)(q+1)^2} P_\pi\left(\frac{\fb_1\fb_2}{q^3}\right), \]
  where $P_\pi(X) = (1 - \alpha X) \dots (1 - \delta X)$. This is equivalent to Corollary 3.10.5 of \emph{op.cit.} (the trace-compatible system $\underline{\phi}_{1, \infty}$ used there corresponds to $\left(\tfrac{q+1}{q}\right)^2 W_{\fa_1}'[\infty] \otimes W_{\fa_2}'[\infty]$ in our present notation).

 \subsection{A \texorpdfstring{``mod $q-1$''}{mod (q-1)} norm relation for \texorpdfstring{$\GSp_4\times \GL_2$}{GSp(4) x GL(2)}}

  \begin{proposition}
   For $n \ge 1$, we have
   \[ (q^2-1)^2 \cdot \fz(\eta t_{1, 1} W_0^{\sph}, W'_{\fa_1}[n], J t_1 W_2^{\sph}) \equiv \frac{-1}{\fb_1^3 \fa_2 \fb_2} P_{(\pi\times \sigma_2)}\left(\fb_1\right) \bmod (q-1), \]
   where the congruence is understood in $\ZZ[\tfrac{1}{q},\alpha^{\pm 1}, \dots, \fa_2^{\pm 1}]$.
  \end{proposition}

  \begin{proof}
   The formulae above tell us that the left-hand side is given by the sum
   \[\sum_{w_0, w_2} \left(\frac{ \alpha^2 \beta \fa_2}{q^5} \cdot \frac{\cE}{\Delta_0 \Delta_2}\right)^{w_0, w_2}. \]

   The Euler factor $\cE$ is given by the product of a subset of 8 among the 16 linear factors appearing in the central $L$-value $L(\pi \times \sigma_1 \times \sigma_2, \tfrac{1}{2})$. The factor $P_{(\pi\times \sigma_2)}\left(\tfrac{\fb_1}{q^3}\right)$ corresponds to a different selection of 8 elements of this set. However, the difference between these two sets ``goes away'' modulo $q-1$: we have identities such as
   \[ \left(1 - \tfrac{\delta\fa_1\fa_2}{q^3}\right) = \left(1 - \tfrac{q^2}{\alpha \fb_1\fb_2}\right) = \frac{-q^2}{\alpha \fb_1\fb_2} \left(1 - \tfrac{\alpha \fb_1\fb_2}{q^2}\right) \equiv \tfrac{-1}{\alpha \fb_1\fb_2}\left(1 - \tfrac{\alpha \fb_1\fb_2}{q^3}\right) \bmod (q-1).\]
   The upshot is that
   \[ \cE \cong \left(\tfrac{-1}{\alpha \fb_1\fb_2}\right)\left(\tfrac{-1}{\alpha \fb_1\fa_2}\right)\left(\tfrac{-1}{\beta \fb_1\fa_2}\right)
   P_{(\pi\times \sigma_2)}\left(\fb_1\right) \bmod q-1.\]
   So, modulo $q-1$ the expression $\frac{ \alpha^2 \beta \fa_2 \cE}{q^5}$ becomes $\frac{-P_{(\pi\times \sigma_2)}\left(\fb_1\right)}{\fb_1^3 \fa_2 \fb_2}$, which is invariant under the Weyl elements $(w_0, w_2)$ and thus factors out of the sum. Since $\sum_{(w_0, w_2)} \tfrac{1}{\Delta_0 \Delta_2} = 1$ the result follows.
  \end{proof}

  \begin{remark}
   This is a much weaker result than that of \cite{HJS20}, who actually write down an explicit linear combination of elements satisfying suitable integrality conditions which give the Euler factor exactly. However, the above weaker statement is sufficient for applications to tame norm relations for Euler system classes, since one always chooses the auxiliary primes such that $q-1$ is highly congruent to 1 modulo the residue characteristic of the Galois representation. As its proof is much less laborious than \emph{op.cit.}, and would generalise to any case where a formula analogous to Gejima's is available, we hope that it may be of independent interest.
  \end{remark}

\section{Klingen-type computations}

 We now consider a second family of formulae, which are a little different in nature: these arise when considering interpolation of coherent cohomology (rather than Betti or \'etale cohomology). It should, in principle, be possible to deduce these formulae from the general result in \cref{prop:pstabShintani}, but it appears to be easier to use an alternative method following \S 8 of \cite{LPSZ1}.

 \subsection{The Klingen Levi}

  \begin{definition}
   Let $\rho$ denote the $\GL_2$ representation with Hecke parameters $\alpha, \beta$; and let $\theta$ be the unramified character mapping a uniformizer to $\tfrac{\gamma}{\alpha} = \tfrac{\delta}{\beta}$, so that $\chi_{\pi} = |\cdot|^2\chi_\rho \theta $ and
   \[ L(\pi, s) = L(\rho, s+1) L(\rho \otimes \theta, s+1).\]
  \end{definition}

  As in \cite[\S 8.4]{LPSZ1}, we see that $\pi$ is isomorphic to the normalised induction of the representation $\rho |\cdot| \boxtimes \theta$ of the Klingen Levi subgroup (identified with $\GL_2 \times \GL_1$ as in \emph{op.cit.}). In particular, $\Hom_{\GL_2}\left(\pi, \rho\right)$ is 1-dimensional, where $\GL_2$ is interpreted as a subgroup of $G$ via $A \mapsto \begin{smatrix}1\\ &A \\&&\det(A)\end{smatrix}$.

  \begin{definition}
   We write $\cL$ for the homomorphism of $\GL_2$-representations $\cW(\pi) \to \cW(\rho)$ normalised such that the following diagram commutes:
   \[
    \begin{tikzcd}
     \pi\dar \rar &\cW(\pi) \dar[dotted, "\cL" right]\\
     \rho \rar & \cW(\rho)
    \end{tikzcd}
   \]
   Here the left-hand vertical arrow is the canonical map between induced representations (restricting functions in $\pi$ from $N_G(F) \backslash G(F)$ to $N_{\GL_2}(F) \backslash \GL_2(F)$), and the horizontal maps are the unipotent integral transforms sending $\pi$ and $\rho$ to their Whittaker models.
  \end{definition}

  One can compute, using the Casselmann--Shalika formula for $G$, that the top horizontal arrow sends the normalised spherical function $\phi_{\sph}$ in the induced representation to the Whittaker function
  \begin{equation}
   \label{eq:CasShaG}
   W_{\phi^{\sph}} =
   \left(1 - \tfrac{\beta}{q\alpha}\right)\left(1 - \tfrac{\gamma}{q\alpha}\right)\left(1 - \tfrac{\gamma}{q\beta}\right)\left(1 - \tfrac{\delta}{q\alpha}\right) W_0^{\sph}.
  \end{equation}
  Similarly, the normalised spherical vector $\xi_{\sph}$ of $\rho$ maps to $W_{\xi_{\sph}} = \left(1 - \tfrac{\beta}{q\alpha}\right) \cW_{\rho}^{\sph}$. Comparing these formulae, we note that $\cL$ sends $W_0^{\sph}$ to $\left[ \left(1 - \tfrac{\gamma}{q\alpha}\right)\left(1 - \tfrac{\gamma}{q\beta}\right)\left(1 - \tfrac{\delta}{q\alpha}\right)\right]^{-1} \cW_{\rho}^{\sph}$.

  \begin{proposition}
   We have $\cL\left(W^{\prime, \Kl}_{\alpha\beta}[r]\right) = q^{3r} W_{\rho}^{\sph}$, for all $r \ge 1$.
  \end{proposition}

  \begin{proof}
   Let $\phi_r$ denote the unique $\Kl(\varpi^r)$-invariant vector in $\pi$ which is supported on $P_{\Kl} \cdot \Kl(\varpi^r)$ and satisfies $\phi_r(1) = q^{3r}$. Then the restriction of $\phi_r$ to $\GL_2(F)$ is obviously $q^{3r} \xi_{\sph}$, whose image in $\cW(\rho)$ is $q^{3r} \left(1 - \tfrac{\beta}{q\alpha}\right) W_{\rho}^{\sph}$.

   On the other hand, the image of $\phi_r$ in $\cW(\pi)$ is a $U_2'$-eigenvector at $\Kl(\varpi^r)$ level with the same eigenvalue as $W^{\prime, \Kl}_{\alpha\beta}[r]$. So the two must be linearly dependent, and we can compare them by comparing their normalised traces to prime-to-$p$ level. The trace of $W_{\phi_r}$ is $\frac{q^3}{1 + q + q^2 + q^3} W_{\phi^{\sph}}$, while the trace of $W^{\prime, \Kl}_{\alpha\beta}$ is given by Proposition \ref{prop:klingentrace}. Comparing this with \eqref{eq:CasShaG} we conclude that $W_{\phi_r} = \left(1 - \tfrac{\beta}{q\alpha}\right) W^{\prime, \Kl}_{\alpha\beta}[r]$ and the result follows.
  \end{proof}

  \begin{proposition}
   We have
   \[ \cL\left(W^{\prime, \Iw}_{\alpha\beta}[r]\right) = q^{3r} W'_{\rho, \alpha}[r]\]
   for all $r \ge 1$.
  \end{proposition}

  \begin{proof}
   This can be proved directly, by a similar argument to the above, replacing $\phi_r$ with an eigenvector supported in $B(F) \cdot \Iw(\varpi^r)$; alternatively, one can note that the map $\cL$ intertwines the Hecke operator $U_{1, \Iw}'$ on $\cW(\pi)^{\Iw_G(\varpi^r)}$ with the operator $U'$ on $\rho^{\Iw_{\GL_2}(\varpi^r)}$, so the result follows formally by applying the operator $\tfrac{1}{\alpha}\left(U_{1, \Iw}' - \beta\right)$ to the previous proposition.
  \end{proof}
 \subsection{Trace-compatible systems}

  \begin{definition}
   For $r \ge 1$, let $\Kl^1(\varpi^r)$ be the group $\{ g \in G(\cO): g \bmod \varpi^r \in N_{\Kl}(\cO/\varpi^r)\}$.
  \end{definition}

  \begin{proposition}
   Let $W \in \cW(\rho)$ be stable under $1 + \varpi^r M_2(\cO)$. Then for all $n \ge r$, there exists a unique $\cL_n^*(w) \in \cW(\pi)$ stable under $\Kl^1(\varpi^n)$ and lying in the $U_{2}' = \tfrac{\alpha\beta}{q}$ eigenspace such that $\cL\left( \cL_n^*(w) \right) = q^{3n} W$. Moreover, for $n \ge r$ the vectors $\cL_n^*(w)$ are compatible under the normalised trace maps.
  \end{proposition}

  \begin{proof}
   This is a version of Casselman's second adjunction formula for parabolically induced representations.
  \end{proof}

  \begin{definition}
   Let us define
   \[ u_{\Kl} = \begin{smatrix} 1 \\ 1 & 1 \\ && 1\\ &&-1&1 \end{smatrix}.\]
   (This element is denoted $\gamma$ in \cite{LPSZ1}, but this notation conflicts with the use of $\gamma$ for a Hecke parameter.)
  \end{definition}

  \begin{proposition}\label{prop:torusintegral}
   Let $W_i \in \cW(\sigma_i)$ for $i = 1, 2$; and let $W_\rho \in \cW(\rho)$. Then the sequence
   \[ \fz\left( u_{\Kl} \cdot \cL_n^*(W_\rho), W_1, W_2\right)\]
   is independent of $n$ for $n \gg 0$; and its limiting value is given by
   \[ \fz\left( u_{\Kl} \cdot \cL_\infty^*(W_\rho), W_1, W_2\right) = \frac{q^3\cB_{\Kl} }{(q+1)^2(q-1)} \int_{F^\times} W_{\rho}(\stbt x 0 0 1)W_{1}(\stbt x 0 0 1)W_{2}(\stbt x 0 0 1)\left(\tfrac{q\gamma}{\alpha}\right)^{v(x)}\, \mathrm{d}x,\]
   where $\cB_{\Kl}$ is the Euler factor defined by
   \[
    \cB_{\Kl} =
    \left.
    \middle( 1 - \tfrac{q^2}{\alpha \fa_1 \fa_2}\middle)
    \middle( 1 - \tfrac{q^2}{\alpha \fb_1 \fa_2}\middle)
    \middle( 1 - \tfrac{q^2}{\alpha \fa_1 \fb_2}\middle)
    \middle( 1 - \tfrac{q^2}{\alpha \fb_1 \fb_2}\middle)
    \middle( 1 - \tfrac{q^2}{\beta \fa_1 \fa_2}\middle)
    \middle( 1 - \tfrac{q^2}{\beta \fa_1 \fb_2}\middle)
    \middle( 1 - \tfrac{q^2}{\beta \fb_1 \fa_2}\middle)
    \middle( 1 - \tfrac{q^2}{\beta \fb_1 \fb_2}\middle)
    \right..
   \]
  \end{proposition}
%

  \begin{remark}
   Note that $\cB_{\Kl}$ has a much larger symmetry group than our previous Euler factors: it is unchanged if we swap $\alpha$ and $\beta$, or if we swap $\fa_i$ and $\fb_i$ for $i = 1$ or $2$.
  \end{remark}

  \begin{proof}
   If $W_\rho$ is the spherical vector of $\rho$, then this is exactly \cite[Proposition 8.14]{LPSZ1} (modulo a minor correction: the volume factor $\frac{q^3}{(q+1)^2(q-1)}$ was omitted in the computations of \emph{op.cit.}). This is proved in \emph{op.cit} under an additional assumption on $W_1$ -- that it be of the form $W^{\Phi}$ for some Schwartz function $\Phi$ with $\Phi(0, 0) = 0$ -- but this is harmless, since for generic values of the parameters every $W_1 \in \cW(\sigma_1)$ has this form.

   A close inspection of the proof shows that the assumption that $W_\rho$ be the spherical vector is never used: the arguments remain valid for any trace-compatible family of vectors in the $\GSp_4$ representation $\pi$ which are supported in the ``small cell'' $P_{\Kl}(F)\Kl(\varpi^r)$.
  \end{proof}
%
%

 \subsection{Particular cases}

  We now record some special cases of the above formula.

  \subsubsection{Klingen eigenvectors (a consistency check)}

   \begin{proposition}
    The elements $u_{\Kl}$ and $J \eta J$ represent the same orbit in $\overline{B}_H \backslash G / P_{\Kl}$ (the unique open orbit).
   \end{proposition}

   Note that $u_{\Kl}$ is \emph{not} in the open $\overline{B}_H$-orbit on $G / B_G$, so this is a specifically ``Klingen'' statement. Since the vectors $t_{x_i} v_{\fa_i}$, for $x_i \ge 1$, are invariant under $\overline{B}_H(\Zp)$, we deduce that for $x_i \ge 1$ and $n \ge 0$, we have
   \begin{equation}
    \label{eq:foo}
    \fz\left(J \eta J s_{0, n} \cdot v^{\prime, \Kl}_{\alpha\beta}, t_{x_1} v_{\fa_1}, t_{x_2} v_{\fa_2} \right) = \fz\left(u_{\Kl} s_{0, n} \cdot v^{\prime, \Kl}_{\alpha\beta}, t_{x_1} v_{\fa_1}, t_{x_2} v_{\fa_2} \right).
   \end{equation}

   \begin{proposition}
    \label{prop:Klingen-eigen}
    For $n, x_1, x_2 \ge 0$, we have
    \[ \fz\left(u_{\Kl} s_{0, n} \cdot W^{\prime, \Kl}_{\alpha\beta}, t_{x_1} v_{\fa_1}, t_{x_2} v_{\fa_2} \right) = \frac{q^3}{(q+1)^2(q-1)} \left(\frac{\alpha\beta}{q^5}\right)^n \left(\frac{\fa_1}{q}\right)^{x_1} \left(\frac{\fa_2}{q}\right)^{x_2} \cdot \cE_{\Kl}, \]
    where
    \begin{align*}
     \cE_{\Kl} &=
     \left.
     \middle( 1 - \tfrac{q^2}{\alpha \fa_1 \fa_2}\middle)
     \middle( 1 - \tfrac{q^2}{\alpha \fb_1 \fa_2}\middle)
     \middle( 1 - \tfrac{q^2}{\alpha \fa_1 \fb_2}\middle)
     \middle( 1 - \tfrac{q^2}{\beta \fa_1 \fa_2}\middle)
     \middle( 1 - \tfrac{q^2}{\beta \fa_1 \fb_2}\middle)
     \middle( 1 - \tfrac{q^2}{\beta \fb_1 \fa_2}\middle)
     \right.\\
     &= \left[ \left(1 - \tfrac{q^2}{\alpha \fb_1\fb_2}\right)\left(1 - \tfrac{q^2}{\beta \fb_1\fb_2}\right)\right]^{-1} \cB_{\Kl}\\
     &= \left[\left(1 - \tfrac{q^2}{\alpha \fb_1\fb_2}\right)\left(1 - \tfrac{q^2}{\gamma \fa_1\fa_2}\right)\right]^{-1} \cE.
    \end{align*}
   \end{proposition}

   \begin{proof}[First proof]
    For $x_1, x_2 \ge 1$, we can compute $\fz\left(J \eta J s_{0, n} \cdot v^{\prime, \Kl}_{\alpha\beta}, t_{x_1} v_{\fa_1}, t_{x_2} v_{\fa_2} \right)$ using \cref{cor:pstabvariants} and the relation
    \[
     v^{\prime, \Kl}_{\alpha\beta} = \frac{1}{\left(1 - \frac{\beta}{\alpha}\right)} v^{\prime, \Iw}_{\alpha\beta} + \frac{1}{\left(1 - \frac{\alpha}{\beta}\right)} v^{\prime, \Iw}_{\beta\alpha}.
    \]
    This then gives the proposition by \eqref{eq:foo}.

    In the edge cases where $x_1$ or $x_2$ (or both) is 0, we use the action of the Hecke operator $U$ on $v_{\fa_i}$, noting that one can commute a coset representative $\stbt 1 * 0 1$ for $\star \in \cO$ past $u_{\Kl} s_{0, n}$. (This step would not work with $J \eta J$ in place of $u_{\Kl}$, so it is not clear if \eqref{eq:foo} is valid in this case.)
   \end{proof}

   \begin{proof}[Second proof]
    We now give a proof via \cref{prop:torusintegral}. The value $\left(\frac{q^5}{\alpha\beta}\right)^n\fz\left(u_{\Kl} s_{0, n} \cdot v^{\prime, \Kl}_{\alpha\beta}, t_{x_1} v_{\fa_1}, t_{x_2} v_{\fa_2} \right) = \fz\left(u_{\Kl} \cdot v^{\prime, \Kl}_{\alpha\beta}[n+1], t_{x_1} v_{\fa_1}, t_{x_2} v_{\fa_2} \right)$ is easily seen to be independent of $n \ge 0$, so we may assume $n$ is large enough to apply \cref{prop:torusintegral}. If $x = \varpi^k$, then we have $W_{\rho}^{\sph}(\stbt{x}{}{}{1}) = (\alpha^{k+1} - \beta^{k+1})/q^k(\alpha - \beta)$, so the torus integral is $\left(\tfrac{\fa_1}{q}\right)^{x_1} \left(\tfrac{\fa_2}{q}\right)^{x_2}\left[ \left(1 - \frac{q^2}{\alpha \fb_1\fb_2}\right)\left(1 - \frac{q^2}{\beta \fb_1\fb_2}\right)\right]^{-1}$.
   \end{proof}

  \subsubsection{Iwahori eigenvectors}

   \begin{definition}
    Let $u_{\Iw} = u_{\Kl} w_1$, where $w_1 = \begin{smatrix}1 \\ &&1\\&-1\\&&&1\end{smatrix}$; that is, $u_{\Iw} =
    \begin{smatrix}1 \\1 & & 1\\ &-1\\&\phantom{-}1&&1\end{smatrix}$.
   \end{definition}

   \begin{proposition}
    For $n \ge 1$ we have
    \[
     \fz\left( u_{\Iw} W_{\alpha\beta}^{\prime, \Iw}[n], W_{\fa_1}, W_{\fa_2} \right) =
     \frac{q^3}{(q+1)^2(q-1)} \left(1 - \tfrac{q^2}{\beta\fb_1\fb_2}\right)^{-1} \mathbf{B}_{\Kl}.
    \]
   \end{proposition}

   \begin{proof}
    Again, the value is clearly independent of $n \ge 1$ so we may assume $n$ to be large. We have $q^{-3n} \cL_n\left(w_1 \cdot W_{\alpha\beta}^{\prime, \Iw}[n]\right) = \stbt{}{1}{-1}{} W_{\alpha}'[n] = \left(\tfrac{q}{\alpha}\right)^n \stbt{\varpi^n}{}{}{1} W_\alpha$, whose value at $\stbt{\varpi^k}{}{}{1}$ is $(\alpha/q)^k$ for $k \ge -n$. So the formula follows easily from Proposition \ref{prop:torusintegral}.
   \end{proof}

 \subsubsection{Depleted vectors}

  \begin{definition}
   For $i = 1, 2$, we define the \emph{depleted Whittaker function} of $\sigma_i$ by
   \[ W^{\dep}_i = \left(1 - \tfrac{1}{\fa_i}\stbt{1}{}{}{\varpi} \right)
   \left(1 - \tfrac{1}{\fb_i}\stbt{1}{}{}{\varpi} \right)W_i^{\sph} \in \cW(\sigma_i)^{\Iw(\fp^2)}\]
  \end{definition}

  Note that this is symmetric in $\fa_i$ and $\fb_i$, and it is normalised to 1 at the identity; in fact we have $W_i^{\dep}(\stbt{x}{}{}{1}) = \ch_{\cO^\times}(x)$.


  \begin{theorem}
   All of the following quantities are equal to $\frac{q^3}{(q+1)^2(q-1)}\mathbf{B}_{\Kl}$, for any $n \ge 2$:
   \begin{multline*}
    \fz\left(u_{\Kl} \cdot v^{\prime, \Kl}_{\alpha\beta}[n], v_1^{\dep}, v_2^{\dep}\right) = \fz\left(u_{\Kl} \cdot v^{\prime, \Kl}_{\alpha\beta}[n], v_{\fa_1}, v_2^{\dep}\right) = \fz\left(u_{\Kl} \cdot v^{\prime, \Kl}_{\alpha\beta}[n], v_1^{\dep}, v_{\fa_2}\right) \\ =
   \fz\left(u_{\Iw} \cdot v^{\prime, \Iw}_{\alpha\beta}[n], v_1^{\dep}, v_2^{\dep}\right) = \fz\left(u_{\Iw} \cdot v^{\prime, \Iw}_{\alpha\beta}[n], v_{\fa_1}, v_2^{\dep}\right) = \fz\left(u_{\Iw} \cdot v^{\prime, \Iw}_{\alpha\beta}[n], v_1^{\dep}, v_{\fa_2}\right).
   \end{multline*}
  \end{theorem}

  \begin{proof}
   If we take any of the above vectors for $W_1$ and $W_2$, then the torus integral in Proposition \ref{prop:torusintegral} for these vectors reduces to $W_{\rho}\left(1\right)$, where $W_{\rho}$ is $W_{\rho}^{\sph}$ for the first three integrals, and $\left(\tfrac{q}{\alpha}\right)^n \stbt{\varpi^n}{}{}{1} W_\alpha$ for the remaining three. Since both are normalised to 1 at the identity, the result follows.
  \end{proof}

\section{Applications to Euler systems for \texorpdfstring{$\GSp_4$}{GSp(4)}}

 \subsection{Notation changes}

  We now shift notation somewhat to match \cite{LSZ17} and \cite{LZ20}. We suppose $F = \Qp$, and relabel the Hecke parameters as follows:
  \[
  \begin{array}{c|c}
   \alpha,\dots,\delta & p^{-q} \alpha, \dots, p^{-q}\delta\\
   \fa_i, \fb_i & p, p^{-t_i} \chi_i(p)^{-1}
  \end{array}
  \]
  where $t_1 = r_1 - q-r$, $t_2 = r_2 - q + r$, and $\chi_1$ and $\chi_2$ are finite-order characters with $\chi_1 \chi_2 = \chi_{\pi}$. (This $q$ is an arbitrary integer, corresponding to the $L$-value $L(V_\pi, 1 + q)$, with the geometric range being $0 \le q \le r_2$; it is not the same $q$ as before, which is now $p$.)

 \subsection{Siegel sections}

  Let $\cS(\Qp^2, \CC)$ be the space of Schwartz functions on $\Qp^2$. Then (for the above values of the parameters) we have $\GL_2(F)$-equivariant maps $\cS(\Qp^2, \CC) \to \cW(\sigma_i)$, $\Phi \mapsto W_i^{\Phi}$. These are defined as in \cite[\S 8.1]{LPSZ1}; the normalisations are such that $\ch(\Zp^2)$ maps to $W_i^{\sph}$.

  Our $p$-stabilised Schwartz functions can then be described as follows:

  \begin{itemize}
  \item If $\Phi= \ch(\Zp^\times \times \Zp)$, then $W_i^{\Phi} = W_{\fa_i}$, where $\fa_i = p^{-t_i} \chi_i(p)^{-1}$.
  \item If $\Phi = \ch(p^n\Zp \times \Zp^\times)$, for $n \ge 1$, then $W^{\Phi}_i = p^{-n} W'_{\fa_i}[n]$, for this same $\fa_i$.
  \end{itemize}

  It's convenient to introduce the notation $\Phi[n]$ for the Schwartz function
  \[ \left(\, (p^2 - 1)p^{2n-2} \cdot \ch( (0, 1) + p^n \Zp)\, \right)_{n \ge 1},\]
  which are invariant under $K_{H, 1}(p^n) = \{\stbt{\star}{\star}{}{1} \bmod p^n\}$, and compatible under the normalised trace maps between these levels for $n \ge 1$. We then have
  \[ W_i^{\Phi[n]} = \tfrac{p+1}{p} W'_{\fa_i}[n].\]

  \subsection{Relating parahoric and spherical classes} With these modified notations, we can rewrite our Siegel-level computation \cref{eq:siegel} as
  \begin{multline*}
   \fz\left( \eta t_{1, 0} \cdot U_1^{-1} \cdot W_\alpha^\Si, \Phi[n], \Phi[n]\right) =\\ \frac{1}{p^{2} (p-1)}
   \left(1 - \tfrac{p^q}{\alpha}\right)\left(1 - \tfrac{\beta}{p^{1 + q}}\right)
   \left(1 - \tfrac{\gamma}{p^{1 + q}}\right)
   \left(1 - \tfrac{\delta}{p^{1 + q}}\right)
   \left(1 - \tfrac{p^{r_2 +1 + r}\chi_2(p)}{\alpha}\right)
   \left(1 - \tfrac{\delta}{p^{r_2 + 2 + r}\chi_2(p)}\right)
  \end{multline*}
  for all $n \ge 1$. This is the Euler factor appearing in \cite[Theorem 17.2.2]{LZ20} -- the factor that arises when varying the parameter ``$q$'' in families -- scaled by the volume factor $p^2(p-1)$, which is exactly the index
  \[ \left[ K_{H, 1}(p^\infty): K_{H, 1}(p^\infty) \cap \eta t_{1, 0} \cdot \Si(p) \cdot (\eta t_{1, 0})^{-1}\right].\]

  Similarly, from \eqref{eq:iwahori2} we obtain
  \begin{multline*}
   \fz\left(\eta t_{1, 1} \cdot U_1^{-1}U_2^{-1}\cdot W^{\Iw}_{\alpha\beta}, \Phi[n], \Phi[n]\right) =  \frac{1}{p^{5} (p-1)^2}
    \left(1 - \tfrac{p^q}{\alpha}\right)\left(1 - \tfrac{\beta}{p^{1 + q}}\right)
    \left(1 - \tfrac{\gamma}{p^{1 + q}}\right)
    \left(1 - \tfrac{\delta}{p^{1 + q}}\right)\\
    \times
    \left(1 - \tfrac{p^{r_2 +1 + r}\chi_2(p)}{\alpha}\right)
    \left(1 - \tfrac{p^{r_2 +1 + r}\chi_2(p)}{\beta}\right)
    \left(1 - \tfrac{\gamma}{p^{r_2 + 2 + r}\chi_2(p)}\right)
    \left(1 - \tfrac{\delta}{p^{r_2 + 2 + r}\chi_2(p)}\right),
  \end{multline*}
  for $n \gg 0$. This is the Euler factor one needs in order to vary both $q$ and $r$ in families; see \cite{LRZ}. Again, the denominator $p^5(p-1)^2$ has an interpretation as the index of a certain stabiliser, this time in the larger group $H(\Zp) \times \Zp^\times$. The ratio between the two products of Euler factors above, namely $\left(1 - \tfrac{p^{r_2 +1 + r}\chi_2(p)}{\beta}\right)
  \left(1 - \tfrac{\gamma}{p^{r_2 + 2 + r}\chi_2(p)}\right)$, is the ``fudge factor'' appearing in \cite[Proposition 17.3.8]{LZ20}.

 \subsection{Integral test data and tame norm relations}

  Let $U$ be an open compact in $G$. By a \emph{primitive integral test datum} for $\GSp_4$, of level $U$, we mean a triple $\delta = (g, \Phi_{1}, \Phi_{2})$, with $g \in G / U$ and $\Phi_i \in \cS(\Qp^2, \QQ)$, satisfying the following assumption:
  \begin{quotation}
   There is an open compact subgroup $V \subset H$ such that $V \subseteq H \cap g U g^{-1}$, $V$ acts trivially on the $\Phi_i$, and $\Phi_1 \otimes \Phi_2$ takes values in $C \ZZ$, where $C = \frac{1}{\operatorname{vol} V}$.
  \end{quotation}
  (Here $\operatorname{vol}$ is the volume normalised so that $\operatorname{vol}(K_H) = 1$. If $V \subseteq K_H$, as is often the case, then $C = [K_H : V]$.) We define an \emph{integral test datum} to be a formal sum of primitive integral data. Given such a datum, we define $\fz(\delta) = \fz(g W_0^{\sph}, W^{\Phi_1}, W^{\Phi_2})$.

  \begin{remark}
   The integral test data are -- as the name suggests -- the elements which will map to an integral lattice in the Galois cohomology. So the above theorem is the ``local essence'' of the tame norm relation for Euler system classes. See \cite{LSZ-unitary} for further details of this approach to tame norm relations.
  \end{remark}

  \begin{proposition}
   If $C_\infty = [ K_H(p^\infty) : K_H(p^\infty) \cap g U g^{-1}]$, then $C_\infty \cdot (g, \Phi[n], \Phi[n])$ is an integral test datum for all $n \gg 0$.
  \end{proposition}

  \begin{proof}
   If we let $C_n = [ K_{H, 1}(p^n) : K_{H, 1}(p^n) \cap g U g^{-1}]$, then clearly $C_n \cdot (g, \Phi[n], \Phi[n])$ is an integral test datum for all $n \ge 1$, and for $n \gg 0$ we have $C_n = C_\infty$.
  \end{proof}

  \begin{proposition}
   Let $U = \{ g \in K_G: \nu(g) = 1 \bmod p\}$. Then there exists an integral test datum of level $U$ such that $\fz(\delta) = p P_\pi(p^{-1-q})$.
  \end{proposition}

  \begin{proof}
   This is essentially the content of \cite[\S 3]{LSZ17} (see \S\ref{sect:gsp4tame} above for a summary). We consider the test datum
   \[ \delta_n = \left(\mathrm{id}, \Phi[n], \Phi[n]\right) - \left( t_{1, 0}^{-1} \eta t_{1, 0}, \Phi[n], \Phi[n]\right).\]
   One computes that the index $C_\infty$ is $p-1$ for both terms. Thus $(p-1) \delta_n$ is integral for $n \gg 0$ (in fact $n \ge 2$ suffices).

   Moreover, cancelling out $(p/(p+1))^2$ arising from the comparison between the $\Phi[n]$ and $W_{\fa_i}'[n]$, the result of \cref{sect:gsp4tame} is that for $n \ge 2$ we have
   \[ \fz(\delta_n) = \frac{p}{p-1} P_\pi(p^{-1-q}).\]
   Combining these statements gives the result.
  \end{proof}

 \subsection{Klingen-type formulae}

  For the Euler factor appearing in \cite[\S 6.9]{LZ20}, we use instead a Schwartz function $\Phi_{\mathrm{crit}}$, which maps to the normalised Whittaker function $W_{\fa_i}$ when we label the roots of the Hecke polynomial of $\sigma_i$ in the opposite order (i.e.~so that $\fa_i = p$ and $\fb_i = p^{-t_i} \chi_i(p)^{-1}$). Making these substitutions in \cref{prop:Klingen-eigen}, we obtain
  \begin{multline*}
  \fz\left(W_{\alpha\beta}^{\prime,\Kl}, \Phi_{\mathrm{crit}}, \Phi_{\mathrm{crit}}\right) = \\
  \tfrac{p^3}{(p+1)^2(p-1)}\left.
       \middle( 1 - \tfrac{p^q}{\alpha}\middle)
       \middle( 1 - \tfrac{p^q}{\beta}\middle)
       \middle( 1 - \tfrac{p^{1+r_2 + r}\chi_2(p)}{\alpha}\middle)
       \middle( 1 - \tfrac{p^{1+r_2 + r}\chi_2(p)}{\beta}\middle)
       \middle( 1 - \tfrac{\gamma}{p^{2+r_2 + r} \chi_2(p)}\middle)
       \middle( 1 - \tfrac{\delta}{p^{2+r_2 + r} \chi_2(p)}\middle)
       \right.,
  \end{multline*}
  as required (although the normalisation factor $\tfrac{p^3}{(p+1)^2(p-1)}$ was omitted in \emph{op.cit.}, which is harmless, since it is a non-zero scalar; this factor is related to the index of the subgroup $K_{H, \Delta}(p)$ appearing in \cite{LPSZ1}).

\section{Application to \texorpdfstring{$\GSp_4 \times \GL_2$}{GSp(4) x GL(2)}}

 The table of substitutions for $\GSp_4 \times \GL_2$ looks as follows:
 \[
 \begin{array}{c|c}
 \alpha,\dots,\delta & p^{-(r_1 + r_2)/2} \alpha, \dots, p^{-(r_1 + r_2)/2}\delta\\
 \fa_2, \fb_2 & p^{-t_2/2} \fa_2, p^{-t_2/2} \fb_2\\
 \fa_1, \fb_1 & p^{-\tfrac{t_1}{2}} \chi_{\pi}\chi_{\sigma_2}(p)^{-1}, p^{1+\tfrac{t_1}{2}} \\
 \end{array}
 \]
 where now $t_2 + 2 = \ell$ is the weight of the $\GL_2$ cusp form, and varying $t_1$ gives us the cyclotomic twist. With these parameters, we land in the cohomology group corresponding to the $L$-value
 \[
  L(V_\Pi \otimes V_\Sigma, 2 + q), \qquad q = \tfrac{r_1 + r_2 - t_1 + t_2}{2},
 \]
 (note signs!), where $V_\Pi$, $V_\Sigma$ each have smallest Hodge number $0$, so the centre of the $L$-function is at $s = \frac{r_1 + r_2 + t_2 + 5}{2}$. In the analytic normalisation this is $L(\pi \times \sigma, \tfrac{-t_1}{2})$.

 \begin{remark}
  For technical reasons, in the results of \cite{LZ21-erl} we had to specialise to the case when $t_1 + t_2 = r_1 - r_2 - 2$, so that $q = r_2 + 1 + t_2$. However, this is not assumed in \cite{HJS20,LZ20b-regulator}, and is not needed for the formal zeta-integral computations we make here.
 \end{remark}

 \subsection{P-stabilisation of \'etale classes}

  To relate the \'etale classes considered in \cite{HJS20} to their analogues at prime-to-$p$ level, the relevant Euler factor is
  \[
   \fz\left( \eta t_{1, 1} \cdot (U_1 U_2)^{-1} W_{\alpha\beta}^{\Iw}, \Phi[\infty], J t_1 \cdot U^{-1} W_{\fa_2}\right).
  \]
  This is given by \cref{eq:iwahori3} with the parameters relabelled as above (taking $\fa_1$ to be the root of smaller valuation), so we obtain
  \[
   \tfrac{1}{p^5(p+1)(p - 1)^2}\prod \left\{ \left(1 - \tfrac{\lambda}{p^{2 + q}}\right) : \lambda = \beta \fb_2,\gamma \fa_2, \gamma\fb_2, \delta \fa_2,  \delta\fb_2 \right\} \cdot \prod \left\{ \left(1 - \tfrac{p^{1 + q}}{\lambda}\right) : \lambda = \alpha \fa_2, \alpha \fb_2, \beta \fa_2\right\}.
  \]
  One can check that this is consistent with the general formalism of Panchishkin subrepresentations: if the weight of $\pi \times \sigma$ lies in the region where the motivic constructions of \cite{HJS20} apply, and $\pi$ and $\sigma_2$ are ordinary at $p$, then $\{\alpha \fa_2, \alpha \fb_2, \beta \fa_2\}$ are exactly the three Frobenius eigenvalues having valuation $\le 1 + q$, while the other five have valuation $\ge 2 + q$.

\section{Cusp forms for \texorpdfstring{$\GSp_4 \times \GL_2 \times \GL_2$}{GSp(4) x GL(2) x GL(2)}}

 We finally mention applications to self-dual cuspidal automorphic representations of $\GSp_4 \times \GL_2 \times \GL_2$ and diagonal cycles, the setting explored in \cite{LZvista}. Again, we must shift notation slightly:
 \[
 \begin{array}{c|c}
 \alpha,\dots,\delta & p^{-(r_1 + r_2)/2} \alpha, \dots, p^{-(r_1 + r_2)/2}\delta\\
 \fa_1, \fb_1 & p^{-t_2/2} \fa_1, p^{-t_2/2} \fb_1\\
 \fa_2, \fb_2 & p^{-t_2/2} \fa_2, p^{-t_2/2} \fb_2\\
 \end{array}
 \]
 where $\sigma_i$ correspond to holomorphic cusp forms of weight $t_i + 2$ (with $t_i \ge -1$), and the $L$-value we want to study is the central value
 \[ L(\Pi \times \Sigma_1 \times \Sigma_2, \tfrac{1}{2}) = L(V_\Pi \otimes V_{\Sigma_1}\otimes V_{\Sigma_2}, 1 + h), \quad h \coloneqq \tfrac{r_1 + r_2 + t_1 + t_2 + 4}{2}. \]

 Making these substitutions in \cref{prop:Klingen-eigen}, we obtain Proposition 3.9.2 of \cite{LZ20b-regulator}, again up to an unimportant volume factor (to be corrected in a future iteration). On the other hand, making the same substitutions in \cref{eq:iwahori1} gives the Euler factor appearing in Theorem 8.2.5 of \cite{LZvista}.

\newcommand{\noopsort}[1]{}
\providecommand{\bysame}{\leavevmode\hbox to3em{\hrulefill}\thinspace}
\providecommand{\MR}[1]{}
\renewcommand{\MR}[1]{%
 MR \href{http://www.ams.org/mathscinet-getitem?mr=#1}{#1}.
}
\providecommand{\href}[2]{#2}
\newcommand{\articlehref}[2]{\href{#1}{#2}}

\end{document}